\documentclass[11pt,lecno]{amsart}
\usepackage{indentfirst,amssymb,amsmath,amsthm}
\newtheorem{theorem}{Theorem}[section]

\newtheorem{lemma}{Lemma}[section]

\newtheorem{corollary}{Corollary}[section]

\newcommand{\be}{\begin{equation}}
\newcommand{\ee}{\end{equation}}
\newcommand{\bea}{\begin{eqnarray}}
\newcommand{\eea}{\end{eqnarray}}
\newcommand{\beas}{\begin{eqnarray*}}
\newcommand{\eeas}{\end{eqnarray*}}

\begin{document}

\setcounter{page}{1} \setlength{\unitlength}{1mm}\baselineskip
.58cm \pagenumbering{arabic} \numberwithin{equation}{section}

\title[A note on Solitons]{A note on Almost Riemann Soliton and gradient almost Riemann soliton }
\author [K.De and U.C.De]{Krishnendu De and Uday Chand De}
\address {\newline Krishnendu De, \newline Assistant Professor of Mathematics, \newline Kabi Sukanta Mahavidyalaya,
\newline Bhadreswar, P.O.-Angus, Hooghly, \newline Pin 712221, West Bengal, India.}
\email {krishnendu.de@outlook.in}
\address {\newline Uday Chand De \newline Department of Pure Mathematics \newline University of Calcutta
\newline 35, Ballygunge Circular Road \newline Kol- 700019, West Bengal, India.}
\email {uc$\_$de@yahoo.com}

\footnotetext {$\mathbf{AMS\;2010\hspace{5pt}Mathematics\; Subject\;
Classification\;}:$ 53C15, 53C25
53D15.
\\ {Key words and phrases: 3-dimensional normal almost contact metric manifold, Almost Riemann soliton, Gradient almost Riemann soliton.\\
}}
\maketitle

\vspace{.5cm}

\begin{abstract}
The quest of the offering article is to investigate \emph{almost Riemann soliton} and \emph{gradient almost Riemann soliton} in a non-cosymplectic normal almost contact metric manifold $M^3$. Before all else, it is proved that if the metric of $M^3$ is Riemann soliton with divergence-free potential vector field $Z$, then the manifold is quasi-Sasakian and is of constant sectional curvature -$\lambda$, provided $\alpha,\beta =$ constant. Other than this, it is shown that if the metric of $M^3$ is \emph{ARS} and $Z$ is pointwise collinear with $\xi $ and has constant divergence, then $Z$ is a constant multiple of $\xi $ and the \emph{ARS} reduces to a Riemann soliton, provided $\alpha,\;\beta =$constant. Additionally, it is established that if $M^3$ with $\alpha,\; \beta =$ constant admits a gradient \emph{ARS} $(\gamma,\xi,\lambda)$, then the manifold is either quasi-Sasakian or is of constant sectional curvature $-(\alpha^2-\beta^2)$. At long last, we develop an example of $M^3$ conceding a Riemann soliton.
\end{abstract}

\date{}
\vspace{.5cm}
\section{\textbf{Introduction}}
Since Einstein manifolds play out a huge job in Mathematics and material science, the examination of Einstein manifolds and their speculations is an intriguing point in Riemannian and contact geometry. Lately, various generalizations of Einstein manifolds such as Ricci soliton, gradient Einstein soliton, gradient Ricci soliton, gradient $m$-quasi Einstein soliton etc. have researched. The notion of Ricci flow was introduced by Hamilton \cite{hr} and defined by $\frac{\partial}{\partial_t}g(t)=-2S(t)$, where $S$ denotes the Ricci tensor.\par
As a spontaneous generalization, the idea of Riemann flow (\cite{ud1},\cite{ud2}) is defined by $\frac{\partial}{\partial_t}G(t)=-2Rg(t) $, $G=\frac{1}{2}g\otimes g$, where $R$ is the Riemann curvature tensor and $\otimes$ is \emph{Kulkarni-Nomizu} product (executed as (see Besse \cite{be}, p. 47),
$$(P \otimes Q)(E,F,Z,W) =P(E,W)Q(F,U) + P(F,U)Q(E,W)$$
$$-P(E,U)Q(F,W) -P(F,W)Q(E,U)).$$
Analogous to Ricci soliton, the entrancing thought of Riemann soliton was promoted by Hirica and Udriste \cite{hud}. As per Hirica and Udriste \cite{hud}, a Riemannian metric $g$ on a Riemannian manifold $M$ is called a \emph{ Riemann solitons} if there exists a $C^{\infty}$ vector field $Z$ and a real scalar $\lambda$ such that
\begin{equation}\label{aa1}
    2R+ \lambda g\otimes g+ g\otimes \pounds_Z g=0.
\end{equation}

Here we should see that, this new thought of Riemann soliton is nothing but a generalization of the space of constant sectional curvature.
The soliton will be termed as $expanding$ (if $\lambda$ $>$ 0), $steady$ (if $\lambda$ = 0) or $shrinking$ (if $\lambda$ $<$ 0), respectively.
The manifold is said to be \emph{gradient Riemann soliton} if the vector field $Z$ is gradient of the potential function $\gamma$. For this situation the forerunner condition can be composed as
\begin{equation}\label{aa2}
    2R+ \lambda g\otimes g+ g\otimes \nabla ^{2} \gamma=0,
\end{equation}
where $\nabla ^{2} \gamma$ denotes the Hessian of $\gamma$. In the event that we fix the condition on the parameter $\lambda$ to be a variable function, then the equation (\ref{aa1}) and (\ref{aa2}) turns in to \emph{ARS} and \emph{gradient ARS} respectively.  All through this paper the terminology ``almost Riemann solitons" is composed as \emph{ARS}.\par

Riemann solitons and gradient Riemann solitons on Sasakian manifolds have been investigated in detail by Hirica and Udriste (see, \cite{hud}). Furthermore, Riemann's soliton concerning infinitesimal harmonic transformation was studied in \cite{st}. In this association, we notice that Sharma in \cite{sah} explored almost Ricci soliton in $K$-contact geometry and in \cite{sah1}, with divergence-free soliton vector field. In \cite{dev}, Riemann soliton under the context of contact manifold has been studied and demonstrated a few intriguing outcomes.\par
Quite a long while prior, in \cite{oz}, Olszak explored the three dimensional normal almost contact metric (briefly, $acm$) manifolds mentioning several examples. After the citation of \cite{oz}, in recent years normal $acm$ manifolds have been studied by numerous eminent geometers (see, \cite{dm},\cite{dmaa},\cite{dys},\cite{dy} and references contained in those).\par
The above studies motivate us to investigate an \emph{ ARS} and the \emph{gradient ARS} in a $3$-dimensional normal $acm$ manifolds, since $3$-dimensional normal $acm$ manifold covers Sasakian manifold, Cosymplectic manifold, Kenmotsu manifold and Quasi-Sasakian manifold.\par
The forthcoming article is structured as:

In section 2, we reminisce about some facts and formulas of normal $acm$ manifolds, which we will require in later sections. Beginning from Section 3, after giving the proof, we will engrave our main Theorems. After that, we develop an example of a $3$-dimensional normal $acm$ manifold admitting a Riemann soliton. This exposition terminates with a concise bibliography that has been utilized during the formulation of the article.

\section{Preliminaries}

Let $M^3$ be an $acm$ manifold endowed with a triplet of almost contact structure$(\eta ,\xi ,\phi )$. In details, $M^3$ is an odd-dimensional differentiable manifold equipped with a global $1$-form $\eta$ , a unique characteristic vector field $\xi$ and a $(1,1)$-type tensor field $\phi,$ respectively, such that
\begin{equation}\label{a1}
\phi ^2 E=-E+\eta (E) \xi, \hspace{1cm} \eta (\xi )=1, \hspace{1cm} \phi
\xi =0, \hspace{1cm} \eta \circ \phi =0.
\end{equation}
A structure, named \emph{almost complex structure} $\mathcal{J}$ on $M\times \mathbb{R}$ is defined by
\begin{equation}\label{a2}
\mathcal{J}(E,\lambda \frac{d}{ds})=(\phi E-\lambda \xi ,\eta (E)\frac{d}{ds}),
\end{equation}
where $(E,\lambda \frac{d}{ds})$ indicates a tangent vector on $
M\times \mathbb{R}$, $E$ and $\lambda \frac {d}{ds}$ being tangent to $M$ and $\mathbb{R}$
respectively. After fulfilling the condition, the structure $\mathcal{J}$ is
integrable, $M (\eta ,\xi ,\phi )$ is said to be normal (see, \cite{deb},\cite{de}).\\
The \emph{Nijenhuis} torsion is defined by
\begin{equation} \nonumber [\phi, \phi](E,F)=\;\phi ^{2}[E,F]+[\phi E,\phi F]-\phi[\phi E,F]-\phi[E,\phi F].\end{equation}
The structure $(\eta ,\xi ,\phi)$ is said to be normal if and only if
\begin{equation}\label{a3}
[\phi ,\phi ]+2d\eta \otimes \xi=0.
\end{equation}
The Riemannian metric $g$ on $M^3$ is said to be compatible with $(\eta ,\xi ,\phi )$ if the condition
\begin{equation}\label{a5}
g(\phi E,\;\phi F)=\;g(E,F)-\eta (E)\;\eta (F),
\end{equation}

holds for any $E,F\in \mathfrak{X}(M)$. In such case, the quadruple $(\eta,\xi,\phi ,g)$
is termed as an \emph{$acm$ structure} on $M^3$ and $M^3$ is an \emph{$acm$ manifold}. The equation
\begin{equation}\label{a6}
\eta (E)=g(E,\xi ),
\end{equation}
is withal valid on such a manifold.\par
Certainly, we can define the fundamental $2$-form $\Phi $ by
\begin{equation}\label{a7}
\Phi (E,F)=g(E,\phi F),
\end{equation}
where $E, F\in \mathfrak{X}(M)$.\par

For a normal $acm$ manifold, we can write \cite{oz}:
\begin{equation}\label{b1}
(\nabla _E\phi )(F)=g(\phi \nabla _E\xi ,F)-\eta (F)\phi \nabla _E\xi ,
\end{equation}
\begin{equation}\label{b2}
\nabla _E\xi =\alpha [E-\eta (E)\xi ]-\beta \phi E,
\end{equation}
where $\alpha = \frac{1}{2}div \xi$ and $\beta = \frac{1}{2}tr (\phi \;\nabla \xi )$, $div \xi $ is the divergent of $\xi $ defined by $div \xi =trace \{E\longrightarrow \nabla
_E\xi \}$ and $tr(\phi \nabla \xi )=trace \{E\longrightarrow \phi \nabla
_E\xi \}$. Utilizing $(\ref{b2})$ in $(\ref{b1})$ we lead
\begin{equation}\label{b3}
(\nabla _E\;\phi )(F)=\alpha [g(\phi E,\;F)\xi -\eta (F)\;\phi E]+\beta [g(E,F)\;\xi-\;\eta (F)E].
\end{equation}

Also in this manifold the subsequent relations hold \cite{oz}:\newline
\begin{eqnarray}\label{b4}
R(E,F)\xi &=&[F\alpha +(\alpha ^2-\beta ^2)\eta (F)]\phi ^2E \notag \\
&&-[E\alpha +(\alpha ^2-\beta ^2)\eta (E)]\phi ^2F \\
&&+[F\beta +2\alpha \beta \eta (F)]\phi E \notag \\
&&-[E\beta +2\alpha \beta \eta (E)]\phi F, \notag
\end{eqnarray}
\begin{eqnarray}\label{b5}
S(E,\xi )&=&-E\alpha -(\phi E)\beta \\
&&-[\xi \alpha\; +2\;(\alpha ^2-\beta ^2)]\;\eta (E), \notag
\end{eqnarray}
\begin{equation}\label{b6}
\xi \beta\; +2\alpha \;\beta =\;0,
\end{equation}
\newline
\begin{equation}\label{b7}
(\nabla _E\eta )(F)=\alpha g(\phi E,\phi F)-\beta g(\phi E,F).
\end{equation}

It is well admitted that in a $3$-dimensional Riemannian manifold the Riemann curvature tensor is always satisfies
\begin{eqnarray}\label{b8}
R(E,F)Z&=&S(F,Z)E-S(E,Z)F+g(F,Z)Q E-g(E,Z)Q F \\
&&-\frac {r}{2}[g(F,Z)E-g(E,Z)F] . \notag
\end{eqnarray}

By $(\ref{b4})$, $(\ref{b5})$ and $(\ref{b8})$ we infer
\begin{eqnarray}\label{b9}
S(E,F)&=&\;(\frac {r}{2}+\;\xi \alpha +\alpha ^2-\;\beta ^2)g(\phi E,\phi F) \notag \\
&&-\eta (E)(F\alpha +(\phi F)\beta )\;-\eta (F)(E\alpha \;+(\phi E)\beta )\\
&&-2(\alpha ^2\;-\beta ^2)\eta (E)\eta (F). \notag
\end{eqnarray}
For $\alpha ,\beta =$constant, it follows from the above equation that a $3$-dimensional normal $acm$ manifold becomes an $\eta $-Einstein manifold.\par
From $(\ref{b3})$ we conclude that the manifold is either $\alpha $-Kenmotsu \cite{dj} or cosymplectic \cite{deb} or $\beta $-Sasakian, provided $\alpha ,\beta =$constant.
Also it is well known that a $3$-dimensional normal $acm$ manifold reduces to a quasi-Sasakian manifold if and only if $\alpha =0$ (see, \cite{oz},\cite{ols}).\par

\section{Riemann Soliton}

In this segment, we first write the subsequent result (\cite{cho},\cite{cho1}):
\begin{lemma}
In a Riemannian manifold if ($g,Z$) is a Ricci soliton, then we have
\begin{eqnarray} \label{n1}
\frac{1}{2}\lVert {\pounds_Z g}\rVert ^{2}=dr(Z)+2div(\lambda Z-Q Z).
\end{eqnarray}
\end{lemma}

Now, because of (\ref{b2}) we obtain

\begin{equation}\label{n2}
(\pounds_{\xi}g)(E,\;F)=2\alpha\{g(E,\;F)-\eta (E)\;\eta (F)\}.
\end{equation}

We consider a normal $acm$ manifold $M^3$ with $\alpha,\beta=$constants admitting a Riemann soliton defined by(\ref{aa1}). Using \emph{Kulkarni-Nomizu} product in (\ref{aa1}) we write
\begin{eqnarray}\label{cc1}
&&2R(E,F,W,X)+2\lambda \{ g(E,\;X)g(F,W)-g(E,W)g(F,\;X)\}\nonumber\\
&+&\{g(E,X)(\pounds_{Z}g)(F,W)+g(F,W)(\pounds_{Z}g)(U,E)\nonumber\\
&-&g(E,W)(\pounds_{Z}g)(F,X)-g(F,X)(\pounds_{Z}g)(E,W)\}=0.
\end{eqnarray}
Contracting (\ref{cc1}) over $E$ and $X$, we infer

\begin{eqnarray}\label{cc2}
(\pounds_{Z}g)(F,W)+2S(F,W)+(4\lambda+2 div Z)g(F,W)=0.
\end{eqnarray}
Thus Riemann soliton whose potential vector field is of vanishing divergence reduces to Ricci soliton.

Hence we have
\begin{eqnarray}\label{n3}
(\pounds_{Z}g)(F,W)+2S(F,W)+4\lambda g(F,W)=0.
\end{eqnarray}
Setting $Z=\xi$ and utilizing (\ref{n2}) we lead
\begin{eqnarray}\nonumber
2\alpha\{g(F,W)-\eta (F)\eta (W)\}+2S(F,W)+4\lambda g(F,W)=\;0,
\end{eqnarray}
which implies that
\begin{eqnarray}\label{n4}
\alpha\{F-\eta (F)\xi\}+Q F+2\lambda F=0.
\end{eqnarray}
Putting $F=W=e_i$ and taking $div Z=0$ from (\ref{n3}) we get
$r=-6\lambda$.
Hence in our case (\ref{n1}) takes the form
\begin{eqnarray} \label{n5}
\frac{1}{2}\lVert {\pounds_Z g}\rVert ^{2}=dr(Z)+2div(-2\lambda Z-Q Z).
\end{eqnarray}
From (\ref{n4}) we get $Q\xi=-2\lambda \xi$.
Therefore utilizing $r=-6\lambda$ and $Q\xi=-2\lambda \xi$ we obtain from (\ref{n5}) $\xi$ is a Killing vector.
Hence (\ref{n2}) implies $\alpha=0$ that is the manifold is quasi-Sasakian.\par
Utilizing $\alpha=0$ in (\ref{n4}), we infer
$$QF=-2\lambda F,$$. Hence from (\ref{b8}) we can write that the manifold is of constant sectional curvature $-\lambda$. Therefore we write:\\

\begin{theorem} If the metric of a non-cosymplectic normal $acm$ manifold $M^3$ is Riemann soliton with a divergence-free potential vector field, then the manifold is quasi-Sasakian and is of constant sectional curvature $-\lambda$, provided $\alpha,\beta =$constant .\end{theorem}

\section{Almost Riemann Soliton}
 Here we consider a normal $acm$ manifold $M^3$ with $\alpha,\beta=$constants admitting an \emph{ARS} defined by(\ref{aa1}).

In particular, let the potential vector field $Z$ be point-wise collinear with $\xi $ (i.e., $Z=c\xi $, where $c$ is a function on $M$) and has constant divergence. Then
from (\ref{cc2}) we lead

\begin{equation}\label{c1}
g(\nabla _E c\xi,F)+g(\nabla _F c\xi,E)+2S(E,F)+(4\lambda+2div Z) g(E,F)=0.
\end{equation}

Utilizing (\ref{a6}) and (\ref{b2}) in (\ref{c1}), we obtain

\begin{eqnarray}\label{c2}
&&2\alpha c[g(E,F)-\eta(E)\;\eta(F)]+(E c)\eta (F)+(Fc)\eta (E)\\
&&+2S(E,F)+(4\lambda+2div Z) g(E,F)=0.\nonumber
\end{eqnarray}
Replacing $F$ by $\xi $ in (\ref{c2}) and utilizing (\ref{a1}), (\ref{a6}) and (\ref{b5}) gives
\begin{equation}\label{c3}
(E c)+(\xi c)\eta (E)-4(\alpha ^2-\beta ^2)\eta (E)+(4\lambda+2div Z) \eta (E)=0.
\end{equation}

Putting $E=\xi $ in $(\ref{c3})$ and utilizing $(\ref{a1})$ yields
\begin{equation}\label{c4}
\xi c=[2(\alpha ^2-\beta ^2)-2\lambda -2div Z] .
\end{equation}

Putting the value of $\xi c$ in $(\ref{c3})$ we infer

\begin{equation}\label{c5}
dc=[2(\alpha ^2-\beta ^2)-2\lambda -2div Z]\eta .
\end{equation}

Applying $d$ on $(\ref{c5})$ and using Poincare lemma $d^{2}\equiv$0, we lead

\begin{equation}\label{c6}
[2(\alpha ^2-\beta ^2)-2\lambda -2div Z]d\eta +(d\lambda )\eta =0.
\end{equation}

Taking wedge product of $(\ref{c6})$ with $\eta $, we obtain

\begin{equation}\label{c7}
[2(\alpha ^2-\beta ^2)-2\lambda -2div Z]\eta \wedge d\eta =0.
\end{equation}

Since $\eta \wedge d\eta \neq0$ , we infer

\begin{equation}\label{c8}
[2(\alpha ^2-\beta ^2)-2\lambda -2div Z]=0.
\end{equation}

Utilizing $(\ref{c8})$ in $(\ref{c5})$ gives $dc=0$ i.e., $c=$constant. Also from (\ref{c8}) we have

\begin{equation}\label{c9}
\lambda=[(\alpha ^2-\beta ^2) - div Z]=constant.
\end{equation}

Hence we write the following:\\

\begin{theorem} If the metric of a non-cosymplectic normal $acm$ manifold $M^3$ is ARS and $Z$ is pointwise collinear with $\xi $ and has constant divergence, then $Z$ is constant multiple of $\xi $ and the ARS reduces to a Riemann soliton, provided $\alpha ,\beta =$constant .\end{theorem}

\begin{corollary}
If a non-cosymplectic normal $acm$ manifold $M^3$ with $\alpha, \beta =$constant admits an ARS of type ($g,\xi$), then the ARS reduces to a Riemann soliton.
\end{corollary}

\section{Gradient Almost Riemann Soliton}
In this section we investigate a non-cosymplectic normal $acm$ manifold $M^3$ with $\alpha, \beta =$constant, admitting gradient \emph{ARS}. Now we prove the subsequent results:
\begin{lemma}
For a non-cosymplectic normal $acm$ manifold $M^3$ with $\alpha, \beta =$constant, we have
\begin{eqnarray} \label{g1}
(\nabla _EQ)\xi =-\{\frac{r}{2}+3(\alpha ^2-\beta ^2)\}[\alpha \{E-\eta (E)\xi \}-\beta \phi E].\end{eqnarray}
\end{lemma}
\begin{proof}[Proof]
For $\alpha ,\beta =$constants, we get from (\ref{b9})

\begin{equation}\label{gl1}
QF=\{\frac{r}{2}+(\alpha ^2-\beta ^2)\}F-\{\frac{r}{2}+3(\alpha ^2-\beta ^2)\}\eta (F)\xi.
\end{equation}

Differentiating (\ref{gl1}) covariantly in the direction of $E$ and using (\ref{b2}) and (\ref{b7}), we get

\begin{eqnarray}\label{gl2}
(\nabla _E Q)F&=&\frac{dr(E)}{2}(F-\eta(F)\xi)\\
&&-\{\frac{r}{2}+3(\alpha ^2-\beta ^2)\}[\alpha g(E,F)\xi -2\alpha \eta (E)\eta (F)\xi \nonumber\\
&&+\alpha\eta (F)E-\beta g(\phi E,F)\xi -\beta \eta (F)\phi E].\nonumber
\end{eqnarray}

Replacing $F$ by $\xi $ in (\ref{gl2}) and utilizing (\ref{b2}), we get

\begin{equation}\label{gl3}
(\nabla _E Q)\xi =-\{\frac{r}{2}+3(\alpha ^2-\beta ^2)\}[\alpha \{E-\eta (E)\xi \}-\beta \phi E].
\nonumber\end{equation}

\end{proof}
\begin{lemma}
Let $M^3(\eta,\xi,\phi,g)$ be a non-cosymplectic normal $acm$ manifold with $\alpha, \beta =$constant. Then we have
\begin{eqnarray} \label{gl4}
\xi r=-4\alpha\{\frac{r}{2}+3(\alpha ^2-\beta ^2)\}
\end{eqnarray}
\end{lemma}

\begin{proof}[Proof]
Recalling (\ref{gl2}), we can write

\begin{eqnarray}
g((\nabla_{E}Q)F,Z)&=&\frac{dr(E)}{2}[g(F,Z)-\eta(F)\eta (Z)]\\
&&-\{\frac{r}{2}+3(\alpha ^2-\beta ^2)\}[\alpha g(E,F)\eta (Z) -2\alpha \eta (E)\eta (F)\eta (Z) \nonumber\\
&&+\alpha\eta (F)g(E,Z)-\beta g(\phi E,F)\eta (Z) -\beta \eta (F)g(\phi E,Z)].\nonumber
\end{eqnarray}

Putting $E=Z=e_{i}$ (where $\{e_{i}\}$ be the orthonormal basis for the tangent space of $M$ and taking $\sum{i}$, $1\leq i \leq 3$ ) in the foregoing equation and using the so called formula of Riemannian manifolds $divQ=\frac{1}{2}grad$ $r$, we obtain
\begin{eqnarray}
(\xi r)\eta(F)=-4\alpha\{\frac{r}{2}+3(\alpha ^2-\beta ^2)\}\eta(F).
\end{eqnarray}
Replacing $F=\xi$ in the previous equation we have the required result.
\end{proof}

\begin{lemma}\label{blem1}(Lemma. 3.8 of \cite{dev})
For any vector fields $E, F$ on $M^3$, in a gradient ARS $(M,g,\gamma,m,\lambda)$, we infer
\begin{eqnarray}\label{gl12}
R(E,F)D \gamma &=& (\nabla _F Q)E-(\nabla_E Q) F \nonumber\\&&+ \{F(2\lambda+\triangle \gamma)E-E(2\lambda+\triangle \gamma)F\},
\end{eqnarray}
where $\triangle \gamma$ = div $D \gamma$, $\triangle$ is the Laplacian operator.
\end{lemma}

Superseding $F$ by $\xi$ in (\ref{gl12}) and utilizing Lemma 5.1, we get
\begin{eqnarray}\label{g2}
R(E,\xi)D \gamma &=& \frac{dr(\xi)}{2}[E-\eta (E)\xi]-\{\frac{r}{2}+3(\alpha ^2-\beta ^2)\}[2\alpha E-2\alpha \eta (E)\xi-\beta \phi E]\nonumber\\&&+ \{\xi(2\lambda+\triangle \gamma)E-E(2\lambda+\triangle \gamma)\xi\}.
\end{eqnarray}
Then utilizing (\ref{b4}), we infer
\begin{eqnarray}\label{g3}
g(E,(\alpha ^2-\beta ^2)D \gamma + D(2\lambda+\triangle \gamma))\xi &=& \frac{dr(\xi)}{2}[E-\eta (E)\xi]\nonumber\\&&-\{\frac{r}{2}+3(\alpha ^2-\beta ^2)\}[2\alpha E-2\alpha \eta (E)\xi-\beta \phi E]\nonumber\\&&+ \{(\alpha ^2-\beta ^2)(\xi \gamma)+\xi(2\lambda+\triangle \gamma)\}E.
\end{eqnarray}
Executing inner product of the foregoing equation with $\xi$ yields
\begin{eqnarray}\label{g4}
E((\alpha ^2-\beta ^2) \gamma + (2\lambda+\triangle \gamma)) = \{(\alpha ^2-\beta ^2)(\xi \gamma)+\xi(2\lambda+\triangle \gamma)\}\eta (E),
\end{eqnarray}
from which easily we lead
\begin{eqnarray}\label{g5}
d( (\alpha ^2-\beta ^2)\gamma + (2\lambda+\triangle \gamma)) = \{(\alpha ^2-\beta ^2)(\xi \gamma)+\xi(2\lambda+\triangle \gamma)\}\eta,
\end{eqnarray}
where the exterior derivative is denoted by $d$.
From the above equation we conclude that $(\alpha ^2-\beta ^2)\gamma + (2\lambda+\triangle \gamma)$ is invariant along the distribution $\mathcal{D}$ . In other terms, $E((\alpha ^2-\beta ^2)\gamma + (2\lambda+\triangle \gamma))=0$ for any $E \in \mathcal{D}$.
Hence utilizing $(\ref{g4})$ in (\ref{g3}), we get
\begin{eqnarray}\label{g6}
\frac{dr(\xi)}{2}[E-\eta (E)\xi]-\{\frac{r}{2}+3(\alpha ^2-\beta ^2)\}[2\alpha E-2\alpha \eta (E)\xi-\beta \phi E]=0.
\end{eqnarray}
Contracting the previous equation and using (\ref{gl4}), we lead
\begin{eqnarray}\label{g7}
\alpha\{\frac{r}{2}+3(\alpha ^2-\beta ^2)\}=0 .
\end{eqnarray}
Now we split our study in the following cases:\par
Case (i): If $\alpha= 0$, then the manifold reduces to a quasi-Sasakian manifold.\par
case (ii): If $r=-6(\alpha^2-\beta^2)$, then from (\ref{b9}) we get $S=-2(\alpha^2-\beta^2)g$ , that is the manifold is an Einstein manifold and hence from (\ref{b8}) it follows that the manifold is of constant sectional curvature $-(\alpha^2-\beta^2)$.
Hence we write:
\begin{theorem}
If a non-cosymplectic normal $acm$ manifold $M^3$ with $\alpha, \beta =$constant admits a gradient ARS $(\gamma,\xi,\lambda)$, then the manifold is either quasi-Sasakian or is of constant sectional curvature $-(\alpha^2-\beta^2)$.
\end{theorem}

\section{Example }
We consider the manifold $M=\{(x,y,z)\in \mathbb{R}^{3}, z\neq0\}$ and
the linearly independent vector fields$$u_{1}=z\frac{\partial }{\partial x},\hspace{7pt}u_{2}=z\frac{\partial }
{\partial y}
,\hspace{7pt}u_{3}=z\frac{\partial }{\partial z}.$$

The Riemannian metric $g$ is defined by $$g(u_{1},u_{3})=g(u_{1},u_{2})=g(u_{2},u_{3})=0,$$
$$g(u_{1},u_{1})=g(u_{2},u_{2})=g(u_{3},u_{3})=1.$$

Let the 1-form $\eta $ is given by $\eta (E)=g(E,u_{3})$ for any $E\in \mathfrak{X}(M)$ and the tensor field $\phi $ is given by $$\phi (u_{1})=-u_{2},\hspace{7pt}
\phi (u_{2})=u_{1},\hspace{7pt}\phi (u_{3})=0.$$

Then utilizing the linearity of $\phi $ and $g$, we infer $$\eta (u_{3})=1,$$ $$\phi ^{2}E=-E+\eta
(E)u_{3},$$ $$g(\phi E,\phi F)=g(E,F)-\eta (E)\eta (F),$$ for any $E,F\in \mathfrak{X}(M).$
Obviously, the structure $(\eta ,\xi ,\phi ,g)$ admits an $acm$structure on $M^3$ for $u_{3}=\xi $.
Then we lead
\begin{eqnarray} [u_{1},u_{3}]&=&u_{1}u_{3}-u_{3}u_{1}\nonumber\\&=&z\frac{\partial }{\partial x}
(z\frac{\partial }{\partial z})-z\frac{\partial }{\partial z}(z\frac{\partial }{\partial x})\nonumber\\
&=&z^{2}\frac{\partial ^{2}}{\partial x\partial z}-z^{2}\frac{\partial ^{2}}{\partial z\partial x}-
z\frac{\partial }{\partial x}\nonumber\\&=&-u_{1}.\end{eqnarray}
Similarly $$[u_{1},u_{2}]=0\hspace{10pt}and\hspace{10pt} [u_{2},u_{3}]=-u_{2}.$$

Utilizing Koszul's formula for the Riemannian metric g, we can calculate

$$\nabla _{u_{1}}u_{3}=-u_{1},\hspace{10pt}\nabla _{u_{1}}u_{2}=0,\hspace{10pt}
\nabla _{u_{1}}u_{1}=u_{3},$$
$$\nabla _{u_{2}}u_{3}=-u_{2},\hspace{10pt}\nabla _{u_{2}}u_{2}=u_{3},\hspace{10pt}
\nabla _{u_{2}}u_{1}=0,$$
\begin{equation}\nabla _{u_{3}}u_{3}=0,\hspace{10pt}\nabla _{u_{3}}u_{2}=0,\hspace{10pt}
\nabla _{u_{3}}u_{1}=0.\label{f5}\end{equation}
From the above expression it is obvious that the manifold under consideration is a normal $acm$ manifold with $\alpha $, $\beta $=constants, since it satisfies (\ref{b2}) for $\alpha =-1$ and
$\beta =0$ and $\xi =e_{3}.$

It can be easily verified that
$$R(u_{1},u_{2})u_{3}=0,\hspace{10pt}R(u_{2},u_{3})u_{3}=-u_{2},\hspace{10pt}
R(u_{1},u_{3})u_{3}=-u_{1},$$
$$R(u_{1},u_{2})u_{2}=-u_{1},\hspace{10pt}R(u_{2},u_{3})u_{2}=u_{3},\hspace{10pt}
R(u_{1},u_{3})u_{2}=0,$$
$$R(u_{1},u_{2})u_{1}=u_{2},\hspace{10pt}R(u_{2},u_{3})u_{1}=0,\hspace{10pt}
R(u_{1},u_{3})u_{1}=u_{3}.$$

In this example, it is easy to verify that the characteristic vector field $\xi$ has constant divergence and obviously $\pounds_\xi g=0$ . Then equation (\ref{cc1}) reduces to
\begin{eqnarray}\label{h3}
&&2R(E,F)W+2\lambda \{ g(F,W)E-g(E,W)Y\}=0,
\end{eqnarray}
for all vector field $E,F,W$. Also equation (\ref{h3})holds for $\lambda=1$. Thus the manifold under consideration admits a Riemann soliton $(g,\xi,\lambda)$.


\begin {thebibliography}{99}

\bibitem{be} Besse, A., {\itshape Einstein Manifolds}, Springer, Berlin, 1987. https://doi.org/10.1007/978-3-540-74311-8.\\
\bibitem{deb} Blair, D. E., {\it Contact manifolds in Riemannian geometry},
Lecture notes in math., ${\bf 509}$ (1976), Springer-Verlag, Berlin-New York.\\
\bibitem{de} Blair, D. E., {\it Riemannian geometry of contact and symplectic
manifolds}, Progress in Maths., ${\bf 203}$ (2002), Birkh$\ddot a$user Boston, Inc.,
Boston.\\
\bibitem{cho} Cho, J.T., {\it Notes on contact Ricci soliton,} Proc. Edinb. Math. Soc. {\bf 54} (2011), 47-53.\\
\bibitem{cho1} Cho, J.T., {\it Almost contact 3-Manifolds and Ricci solitons,} Int. J. Geom. Methods Mod. Phys. {\bf 10} (2013), 1220022(7 pages).\\
\bibitem{dev}Devaraja, M.N., Kumara, H.A. and Venkatesha, V., {\itshape Riemann soliton within the framework of contact geometry}, Quaestiones Mathematicae, (2020) DOI:10.2989/16073606.2020.1732495\\
\bibitem{dm} De, U. C. and Mondal, A. K., {\it On 3-dimensional normal almost contact metric manifolds satisfying certain curvature conditions}, Commun. Korean Math. Soc., ${\bf 24}$ (2009), $265-275$.\\
\bibitem{dmaa} De, U. C., Turan, M., Yildiz, A. and De, A., {\it Ricci solitons and gradient Ricci solitons on $3$-dimensional normal almost contact metric manifolds}, Publ. Math. Debrecen, {\bf 80} (2012), 127-142.\\
\bibitem{dys} De, U. C., Yildiz, A. and Sarkar A., {\it Isometric immersion of three dimensional quasi-Sasakian manifolds},
Math. Balkanica (N.S.), ${\bf 22}$ (2008), $297-306$.\\
\bibitem{dy} De, U. C., Yildiz, A. and Yalýnýz, A. F., {\it Locally $\phi$-symmetric normal almost contact metric manifolds of dimension $3$}, Appl. Math. Lett., ${\bf 22}$ (2009), $723-727$.\\
\bibitem{dj} Janssen, D. and Vanhecke, L., {\it Almost contact structures and curvature tensors}, Kodai Math. J., ${\bf 4}$ (1981), $1-27$.\\
\bibitem{hr} Hamilton, R. S., {\itshape The Ricci flow on surfaces}, Math. gen. relativ. (Santa Cruz, CA, 1986), 237--262, Contemp. Math. {\bf71}, (1988).\\
\bibitem{hud}Hirica, I.E. and Udriste, C., {\itshape Ricci and Riemann solitons}, Balkan J. Geom. Applications. {\bf 21} (2016), 35-44.\\
\bibitem{oz} Olszak, Z., {\it Normal almost contact manifolds of dimension three}, Ann. Polon. Math., {\bf47} (1986), 41-50.\\
\bibitem{ols} Olszak, Z., {\it On three dimensional conformally flat quasi-Sasakian manifolds},
Period, Math. Hunger., ${\bf 33}$ (1996), $105-113$.\\
\bibitem{sah}Sharma, R., {\itshape Almost Ricci solitons and K-contact geometry}, Monatsh Math. {\bf175} (2014), 621--628.\\
\bibitem{sah1}Sharma, R., {\itshape Some results on almost Ricci solitons and geodesic vector fields,} Beitr. Algebra Geom. {\bf 59} (2018), 289--294.\\
\bibitem{st} Stepanov, S.E. and Tsyganok, I.I., {\itshape The theory of infinitesimal harmonic trans-formations and its applications to the global geometry of Riemann solitons}, Balk. J.Geom. Appl. {\bf 24} (2019), 113-121.\\
\bibitem{ud1} Udri$s$te, C., {\itshape Riemann flow and Riemann wave}, Ann. Univ. Vest, Timisoara. Ser.Mat.-Inf. {\bf 48} (2010), 265-274.\\
\bibitem{ud2} Udri$s$te, C., {\itshape Riemann flow and Riemann wave via bialternate product Riemannian metric}, preprint, arXiv.org/math.DG/1112.4279v4 (2012).\\


\end{thebibliography}
\end{document}